\documentclass[letterpaper,10pt]{article}

\usepackage[margin=0.5in]{geometry}
\usepackage{fullpage}
\usepackage{enumerate}
\usepackage{authblk}
\usepackage[colorinlistoftodos]{todonotes}
\usepackage{verbatim}%comments out anything between \begin{comment} and \end{comment}
\usepackage{section, amsthm, textcase, setspace, amssymb, lineno, amsmath, amssymb, amsfonts, latexsym, fancyhdr, longtable, ulem, mathtools}
\usepackage{epsfig, graphicx, pstricks,pst-grad,pst-text,tikz, tkz-berge,colortbl}
\usepackage{epsf}
\usepackage{graphicx, color}
\usepackage{float}
\usepackage[rflt]{floatflt}
\usepackage{amsfonts}
\usepackage{latexsym,enumitem}
\usetikzlibrary{fit,matrix,positioning}
\usepackage{pdflscape}
\usetikzlibrary{decorations.pathreplacing}
\usepackage{mathrsfs}
\usepackage{makecell}
\usepackage{kbordermatrix}

\definecolor{darkgreen}{rgb}{0, 0.5, 0}

\DeclareMathOperator{\rank}{rank}

\newtheorem{theorem}{Theorem}
\newtheorem{lemma}{Lemma}

\newtheorem{definition}{Definition}

\newtheorem{Ex}{Example}

\newtheorem*{theorem*}{Theorem}
\newtheorem{remark}{Remark}

\newcommand\addvmargin[1]{ \node[fit=(current bounding box),inner ysep=#1,inner xsep=0]{};}
\newcommand{\ind}{{\rm ind \hspace{.1cm}}}
\newcommand{\C}{\mathbb{C}}

\newcommand\blfootnote[1]{%
  \begingroup
  \renewcommand\thefootnote{}\footnote{#1}%
  \addtocounter{footnote}{-1}%
  \endgroup
}

\begin{document}

\title{\bf Classification of Frobenius, two-step solvable Lie poset algebras}
\author{Vincent E. Coll, Jr.$^{*}$, Nicholas Mayers$^{**}$, and Nicholas Russoniello$^{***}$ }
\maketitle

\noindent
\textit{$^*$Dept. of Mathematics, Lehigh University, Bethlehem, PA, USA:  vec208@lehigh.edu \textup(corresponding author\textup)}\\
\textit{$^{**}$Dept. of Mathematics, Lehigh University, Bethlehem, PA, USA:  nwm215@lehigh.edu}\\
\textit{$^{***}$Dept. of Mathematics, Lehigh University, Bethlehem, PA, USA:  nvr217@lehigh.edu}

%\title{Classification of Frobenius, two-step solvable Lie poset algebras}

%\author[*]{Vincent E. Coll, Jr.}
%\author[*]{Nicholas Mayers}
%\author[*]{Nicholas Russoniello}

%\affil[*]{Department of Mathematics, Lehigh University, Bethlehem, PA, 18015}
%\maketitle

\begin{abstract}
\noindent
We show that the isomorphism class of a two-step solvable Lie poset subalgebra of a semisimple Lie algebra is determined by its dimension. We further establish that all such algebras are absolutely rigid. 
\end{abstract}
\blfootnote{\textit{E-mail addresses:} vec208@lehigh.edu (V. Coll), nwm215@lehigh.edu (N. Mayers), nvr217@lehigh.edu (N. Russoniello)}

%\linenumbers

\noindent
\textit{Mathematics Subject Classification 2010}: 17B20, 05E15

\noindent 
\textit{Key Words and Phrases}: Frobenius Lie algebra, deformation

%%%%%%%%%%%%%%%%%%%%%%%%%%%%%%%%%%%%%%%%%%%%%%%%%%%%%%%
\section{Introduction}
%%%%%%%%%%%%%%%%%%%%%%%%%%%%%%%%%%%%%%%%%%%%%%%%%%%%%%%

\textit{Notation}:  Throughout, we assume that \textbf{k} is an an algebraically closed field of characteristic zero -- which we may assume is $\C$.
\bigskip

If $\mathcal{P}$ is a finite poset with partial order $\preceq$, then the associative \textit{poset} (or \textit{incidence}) \textit{algebra} $A=A(\mathcal{P}, \rm{\textbf{k}})$ is the span over \textbf{k} of elements $e_{ij}$, $i\preceq j,$ with multiplication given by setting $e_{ij}e_{j^\prime k}=e_{ik}$ if $j=j^\prime$ and 0 otherwise. One may define the commutator bracket $[a,b] =ab-ba$ on $A=$($A$, $\rm{\textbf{k}})$ to yield the ``Lie poset algebra'' 
$\mathfrak{g}(\mathcal{P})=\mathfrak{g}(\mathcal{P}, \rm{k})$.  If $|\mathcal{P}|=N$, in which case we may assume that $\mathcal{P}=\{1,\hdots,N\}$ with partial order compatible with the linear order, then $A$ and $\mathfrak{g}$ may be viewed as subalgebras of the algebra of all upper-triangular matrices in
$\mathfrak{gl}(N)$. See Figure~\ref{first}, where the left panel illustrates the Hasse diagram of $\mathcal{P}=\{1,2,3,4\}$ with $1\preceq 2\preceq 3,4$. In the right panel, one has the matrix form defining both $A(\mathcal{P})$ and $\mathfrak{g}(\mathcal{P})$  -- each ``generated'' as above by $\mathcal{P}$.  The possible non-zero entries from \textbf{k} are marked by $\ast$'s.   

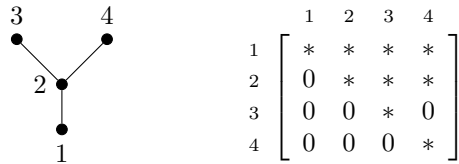
\begin{figure}[H]\label{first}
$$\begin{tikzpicture}[scale=0.6]
	\node (1) at (0, 0) [circle, draw = black, fill = black, inner sep = 0.5mm, label=below:{1}]{};
	\node (2) at (0, 1) [circle, draw = black, fill = black, inner sep = 0.5mm, label=left:{2}] {};
	\node (3) at (-1, 2) [circle, draw = black, fill = black, inner sep = 0.5mm, label=above:{3}] {};
	\node (4) at (1, 2) [circle, draw = black, fill = black, inner sep = 0.5mm, label=above:{4}] {};
	\node (5) at (5, 1.1) {\kbordermatrix{
    & 1 & 2 & 3 & 4  \\
   1 & * & * & * & * \\
   2 & 0 & * & * & * \\
   3 & 0 & 0 & * & 0 \\
   4 & 0 & 0 & 0 & * \\
  }};
    \draw (1)--(2);
    \draw (2)--(3);
    \draw (2)--(4);
\end{tikzpicture}$$
\caption{Hasse diagram of $\mathcal{P}=\{1,2,3,4\}$, $1\preceq 2\preceq 3,4$ and matrix representation }\label{fig:inc} 
\end{figure}

In \textbf{\cite{CG}}, Coll and Gerstenhaber focus on the restriction of  $\mathfrak{g}(\mathcal{P})$ to $A_{|\mathcal{P}|-1}=\mathfrak{sl}(|\mathcal{P}|)$ and denote these algebras by $\mathfrak{g}_A(\mathcal{P})$. They show that both $\mathfrak{g}(\mathcal{P})$ and $\mathfrak{g}_A(\mathcal{P})$ are precisely the subalgebras lying between the Borel subalgebra of upper-triangular matrices and the Cartan subalgebra of diagonal matrices in $\mathfrak{gl}(|\mathcal{P}|)$ and $\mathfrak{sl}(|\mathcal{P}|)$, respectively. The authors further suggest that this model may be used to define Lie poset algebras more generally;  that is, a \textit{Lie poset algebra} is one that lies between a Cartan and associated Borel subalgebra of a simple Lie algebra. 

Following the suggestion of \textbf{\cite{CG}}, Coll and Mayers \textbf{\cite{pBCD}} extend the notion of Lie poset algebra to the other classical families: $B_n=\mathfrak{so}(2n+1)$, $C_n=\mathfrak{sp}(2n)$, and $D_n=\mathfrak{so}(2n)$,  by providing the definitions of posets of types B, C, and D which encode the standard matrix forms of Lie poset algebras in types B, C, and D, respectively.  There are no conditions on the ``generating'' posets in either $\mathfrak{gl}(n)$ or in $\mathfrak{sl}(n)$. This is not true for types B, C, and D (see Definition \ref{def:bcd}). For the exceptional cases, the Lie poset algebra definition still applies -- even if a generating poset may not be present. 

Here, we go one step further and define Lie poset algebras to be subalgebras of a semisimple Lie algebra lying between a Borel and corresponding Cartan subalgebra. In particular, we are concerned with such Lie poset algebras which are Frobenius and two-step solvable. These two-step solvable algebras can exhibit remarkable complexity (see \textbf{\cite{LV}}), but when they are Frobenius (index zero) they are determined up to isomorphism by their dimension (see Theorem~\ref{thm:iso} -- Classification Theorem).  Moreover, we find that all such Frobenius, height-one Lie poset algebras are absolutely rigid (see Theorem~\ref{thm:h2} - Rigidity Theorem).  The Classification and Rigidity theorems constitute the main results of this article.  The latter result may be regarded as an extension, to the other classical types, of a recent result of \textbf{\cite{pA}} which establishes that a type-A Frobenius, Lie poset algebra of derived length less than three is absolutely rigid.

It is important to note that the rigidity theorem of \textbf{\cite{pA}} relies heavily on a recent cohomological result of Coll and Gerstenhaber \textbf{\cite{CG}}, where they establish, in particular, that if $\mathfrak{g}$ is a type-A Lie poset algebra then the space of infinitesimal deformations $H^2(\mathfrak{g}, \mathfrak{g})$, arising from the Chevalley-Eilenberg cochain complex, is a direct sum of three components,

\begin{eqnarray}\label{thm:CoG}
H^2(\mathfrak{g}, \mathfrak{g})= \left( \bigwedge ^{~~~2}\mathfrak{h}^{*} \bigotimes \mathfrak{c} \right) \bigoplus
\left(\mathfrak{h}^{*} \bigotimes H^1(\Sigma, \rm{k})               \right) \bigoplus H^2(\Sigma, \rm{k} ).
\end{eqnarray}

\noindent
Here, $\mathfrak{h}^{*}$ is the linear dual of a Cartan subalgebra of $\mathfrak{g}$, $\mathfrak{c}$ is the center of $\mathfrak{g}$, and $\displaystyle{\Sigma}$ is the nerve of 
$\mathcal{P}$, now viewed as a category.
Accordingly, the cohomology groups on the right side of (\ref{thm:CoG}) are simplicial in nature.

Unfortunately, the spectral sequence argument used to establish (\ref{thm:CoG}) does not extend to the type-B, C, and D cases as the nerves of the generating posets do not provide an appropriate simplicial complex (see Example~\ref{ex:typeBCD}). Broadening the context -- but restricting the focus -- we offer an alternative to (\ref{thm:CoG}). In particular, we establish, via direct calculation, that the second cohomology group of a Frobenius, two-step solvable Lie poset algebra with coefficients in its adjoint representation is trivial (see Theorem~\ref{thm:h2}).

The structure of the paper is as follows. In Section 2 we set the notation and review some background definitions. In Sections 3 and 4, we establish the Classification and Rigidity theorems, respectively.  In the Epilogue we discuss motivations and pose several questions for further study.

%%%%%%%%%%%%%%%%%%%%%%%%%%%%%%%%%%%%%%%%%%%
\section{Preliminaries}
We begin by formally stating the definition of Lie poset algebras.
\begin{definition}
Given a semisimple Lie algebra $\mathfrak{g}$, let $\mathfrak{b}\subset\mathfrak{g}$ denote a Borel subalgebra and $\mathfrak{h}$ its associated Cartan subalgebra. A Lie subalgebra $\mathfrak{p}\subset\mathfrak{g}$ satisfying $\mathfrak{h}\subset\mathfrak{p}\subset\mathfrak{b}$ is called a Lie poset algebra.
\end{definition}
\noindent

%\begin{remark} We use the term  ``Lie poset algebra" to refer to a Lie poset subalgebra of any semi-simple Lie algebra.
%\end{remark}

%\begin{definition}\label{def:derlen}
%5Given a Lie algebra $\mathfrak{g}$, %recall that the derived series $\mathfrak{g}=\mathfrak{g}^0\supset\mathfrak{g}^1\supset\mathfrak{g}^2\supset\hdots$ is defined by $\mathfrak{g}^i=[\mathfrak{g}^{i-1},\mathfrak{g}^{i-1}]$, for $i>0$. For a Lie algebra $\mathfrak{g}$, the least $k$ for which $\mathfrak{g}^k\neq 0$ and $\mathfrak{g}^j=0$, for $j>k$, is called the \textbf{derived length} of $\mathfrak{g}$; if no such $k$ exists, the derived length is taken to be $\infty$. For Lie poset algebras, we refer to the derived length as the \textbf{height}.
%\end{definition}

%\begin{remark}
%In \textup{\textbf{\cite{pBCD}}}, the authors associate a poset to every type-A, B, C, and D Lie poset algebra. Given this association, the height \textup(one less than the largest cardinality of a chain\textup) of a poset is equal to the height of its corresponding Lie poset algebra.
%\end{remark}

The next definition is fundamental to our study.
\begin{definition}

The \textit{index} of a Lie algebra $\mathfrak{g}$ is defined as 
\[\ind \mathfrak{g}=\min_{F\in \mathfrak{g^*}} \dim  (\ker (B_F)),\]

\noindent where $B_F$ is the skew-symmetric \textit{Kirillov form} defined by $B_F(x,y)=F([x,y])$, for all $x,y\in\mathfrak{g}$. The Lie algebra $\mathfrak{g}$ is \textit{Frobenius} if it has index zero.  An index-realizing functional $F$ is one for which the natural map $\mathfrak{g} \rightarrow \mathfrak{g}^*$ is an isomorphism. In such a case $F$ is called a Frobenius functional.

\begin{remark}
Frobenius Lie algebras are of interest to those studying deformation and quantum group theory stemming from their relation to the classical Yang-Baxter equation   
\textup(see \textup{\textbf{\cite{G1,G2}}}\textup).
\end{remark}

\end{definition}

For our work here, we make use of the following, more computational, 
characterization of the index. Let $\mathfrak{g}$ be an arbitrary Lie algebra with basis $\{x_1,...,x_n\}$. The index of $\mathfrak{g}$ can be expressed using the \textit{commutator matrix}, $([x_i,x_j])_{1\le i,j\le n}$, over the quotient field $R(\mathfrak{g})$ of the symmetric algebra $Sym(\mathfrak{g})$ as follows (cf \textbf{\cite{D}}).

\begin{theorem}\label{thm:commat}
$\ind \mathfrak{g}= n-\rank_{R(\mathfrak{g})}([x_i,x_j])_{1\le i, j\le n}.$
\end{theorem}
\noindent
\textit{Proof.}\footnote{We are indebted to A. Ooms who provided the  proof of Theorem \ref{thm:commat} in a private communication.  While the Theorem is widely quoted (and is due to Dixmier), we could not find a detailed proof in the literature.}
Let $M=([x_i,x_j])_{1\le i,j\le n}$ and $M_f=(f([x_i,x_j]))_{1\le i,j\le n}$, for $f\in\mathfrak{g}^*$. Take any $f\in\mathfrak{g}^*$ and $y=\sum_{k=1}^na_kx_k\in \mathfrak{g}$. We have that $y\in\ker(B_f)$ if and only if $f([x_i,y])=0$, for all $i=1,\hdots,n$, if and only if $\sum_{k=1}^na_kf([x_i,x_k])=0$, for all $i=1,\hdots,n$. Thus, $y\in\ker(B_f)$ if and only if the sequence of values $a_1,\hdots,a_n$ define an element in the nullspace of $M_f$. Thus, $\dim(\ker(B_f))=\dim(\ker(M_f))=n-\rank(M_f)$. Therefore, showing $$\rank(M)=\max_{f\in\mathfrak{g}^*} \rank(M_f)$$ will establish the result. 

If $M$ is of rank $r$, then there exists an $r\times r$ submatrix $A$ of $M$ such that $\det(A)\neq 0$. Since $\det(A)$ is a polynomial in $\{x_1,\hdots,x_n\}$ and we can extend any $f\in\mathfrak{g}^*$ to an algebra homomorphism into $\mathbf{k}$, there exists $f_A\in\mathfrak{g}^*$ such that $\det(f_A(A))=f_A(\det(A))\neq 0$; that is, $M_{f_A}$ has an $r\times r$ submatrix with non-zero determinant. Thus, $$\rank(M)\le \max_{f\in\mathfrak{g}^*} \rank(M_f).$$

Conversely, assume $$\rank(M_{f'})=r=\max_{f\in\mathfrak{g}^*} \rank(M_f),$$ for $f'\in\mathfrak{g}^*$. Further, let $A$ be an $r\times r$ submatrix of $M_{f'}$ with non-zero determinant, and which corresponds to the submatrix $A_M$ of $M$. Then $f'(\det(A_M))=\det(f'(A_M))=\det(A)\neq 0$; that is, $\det(A_M)\neq 0$ and $$\rank(M)\ge \max_{f\in\mathfrak{g}^*}\rank(M_f).$$ The claim follows. \hfill$\qed$
\\*

Since Lie poset algebras are evidently solvable, we lastly recall the following basic definition, since we require certain associated terminology for the Classification theorem.

\begin{definition}\label{def:derlen}
Let $\mathfrak{g}$ be a finite-dimensional Lie algebra.  The derived series $\mathfrak{g}=\mathfrak{g}^0\supset\mathfrak{g}^1\supset\mathfrak{g}^2\supset\hdots$ is defined by $\mathfrak{g}^i=[\mathfrak{g}^{i-1},\mathfrak{g}^{i-1}]$, for $i>0$. The least $k$ for which $\mathfrak{g}^k\neq 0$ and $\mathfrak{g}^j=0$, for $j>k$, is called the derived length of $\mathfrak{g}$. The Lie algebra $\mathfrak{g}$ is solvable if its derived series terminates in the the zero subalgebra. If a Lie algebra has finite derived length $k$, then it is $(k+1)$-step solvable.
\end{definition}

\section{Classification}\label{sec:iso}

In this section, we show that any two Frobenius, two-step solvable Lie poset algebras of the same dimension are isomorphic.

%Lie posets algebras have a basis consisting of a basis for the Cartan subalgebra along with eigenvectors of the adjoint action of such elements. 

We require the following lemma to show that a Lie poset algebra inherits a Cartan-Weyl basis from its parent semi-simple Lie algebra. Recall that a Cartan-Weyl basis of a Lie algebra, $\mathfrak{g},$ is defined to be a basis consisting of generators of the Cartan subalgebra, $\mathfrak{h}\subset\mathfrak{g},$ along with root vectors, $E_{\alpha},$ where $\alpha\in\Psi_{\mathfrak{g}}$ -- the root system of $\mathfrak{g}.$ If $\mathfrak{g}$ is given in a Cartan-Weyl basis, the multiplication of basis vectors takes the form $[h_i,h_j]=0$ for all $i,j\in\{1,\dots,\dim \mathfrak{h}\},$ and $[h_i,E_{\alpha}]=\alpha(h_i)E_{\alpha}.$ In fact, the lemma goes further to guarantee that any subalgebra containing the Cartan of the parent semisimple Lie algebra has such a basis.

\begin{lemma}\label{lem:cartan}
Let $\mathfrak{g}$ be an $n+m$-dimensional semi-simple Lie algebra with Cartan subalgebra $\mathfrak{h}$ and Cartan-Weyl basis $\mathscr{B}(\mathfrak{g})$. All Lie subalgebras of $\mathfrak{g}$ containing $\mathfrak{h}$ have a Cartan-Weyl basis. In particular, if $\mathfrak{g}'$ is a Lie subalgebra of $\mathfrak{g}$ and $\mathfrak{h}\subset \mathfrak{g}'$, then there exists a basis, $\mathscr{B}(\mathfrak{g}')$, of $\mathfrak{g}'$ such that $\mathscr{B}(\mathfrak{g}')\subset\mathscr{B}(\mathfrak{g})$.
\end{lemma}
\begin{proof}
Let $\mathscr{B}(\mathfrak{g})=\{h_1,\dots,h_n,w_1,\dots,w_m\}$, where $\{h_1,\hdots,h_n\}$ forms a basis of $\mathfrak{h}$. Since $\mathfrak{h}\subset\mathfrak{g}'$, it suffices to show that if $a=\sum_{i=1}^r a_iw_i\in\mathfrak{g}'$, then at least one of the $w_i$ is contained in $\mathfrak{g}'$; this will imply that all such $w_i$ are contained in $\mathfrak{g}'$ and we can extend $\{h_1,\hdots,h_n\}$ to a basis $\mathscr{B}(\mathfrak{g}')\subset \mathscr{B}(\mathfrak{g})$. Assume that the given $a\in\mathfrak{g}'$ has minimal $r$ such that no summand $w_i$ is contained in $\mathfrak{g}'$; surely $r\ge 2$. If $[h,a]=c_ha$, for all $h\in\mathfrak{h}$, then $a$ and its summands $w_i$ are elements of the same root space of $\mathfrak{g}$; that is, since root spaces of $\mathfrak{g}$ are one-dimensional, $a=cw_i$ for some summand $w_i$ of $a$, a contradiction. Thus, there must exist $h\in\mathfrak{h}$ such that $[h,a]$ is not a multiple of $a$ but is a linear combination of the same summands. This implies that there is a linear combination of $a$ and $[h,a]$ which is not zero and contains no more than $r-1$ of these summands. One of them is consequently already in $\mathfrak{g}'$, a contradiction.
\end{proof}

We are now in a position to address the first of our main results.

\begin{theorem}[Classification Theorem]\label{thm:iso}
Frobenius, two-step solvable Lie poset algebras are isomorphic if and only if the algebras have the same dimension.
\end{theorem}
\begin{proof}
Let $\mathfrak{g}$ be a two-step, solvable Lie poset algebra and let  
$\mathscr{B}(\mathfrak{g})$ be the Cartan-Weyl basis of $\mathfrak{g}$
guaranteed to exist by Lemma~\ref{lem:cartan}.  Let $D(\mathfrak{g})\subset \mathscr{B}(\mathfrak{g})$ consist of the basis elements of $\mathfrak{h}\subset\mathfrak{g}$ and $T(\mathfrak{g})=\mathscr{B}(\mathfrak{g})\setminus D(\mathfrak{g})$. Define the commutator matrix $C(\mathfrak{g})=([x_i,x_j])_{1\le i, j\le n+m}$, where $$\mathscr{B}(\mathfrak{g})=\{x_1,\hdots,x_{n+m}\}$$ with the elements of $D(\mathfrak{g})$ occurring first, followed by the elements of $T(\mathfrak{g})=\mathscr{B}(\mathfrak{g})\backslash D(\mathfrak{g})$. Note that, with the given ordering, $C(\mathfrak{g})$ has the form illustrated in Figure~\ref{fig:h01m}, where $B(\mathfrak{g})$ has rows labeled by elements of $D(\mathfrak{g})$ and columns labeled by elements of $T(\mathfrak{g})$, and $-B(\mathcal{P})^T$ has these labels reversed.
\begin{figure}[H]
$$\begin{tikzpicture}
  \matrix [matrix of math nodes,left delimiter={(},right delimiter={)}]
  {
    0  & B(\mathfrak{g})   \\       
    -B(\mathfrak{g})^T  & 0   \\       };
\end{tikzpicture}$$
\caption{Matrix form of $C(\mathfrak{g})$}\label{fig:h01m}
\end{figure}
\noindent
To see this, since $[\mathfrak{h},\mathfrak{h}]=0$, it suffices to show that $[x,y]=0$, for all $x,y\in T(\mathfrak{g})$. Assume for a contradiction that there exists $z\in \mathfrak{g}$ such that $[x,y]=z\neq 0$, for some $x,y\in T(\mathfrak{g})$. Since $x,y\notin D(\mathfrak{g})$, there exist $h_x,h_y\in \mathfrak{h}$ such that $[h_x,x]=k_1x$ and $[h_y,y]=k_2y$, for $k_1,k_2\neq 0$. Thus, $x,y\in\mathfrak{g}^1$ and $z\in\mathfrak{g}^2$, a contradiction.

Considering Theorem~\ref{thm:commat}, since $\mathfrak{g}$ is Frobenius, $C(\mathfrak{g})$ is equivalent to a diagonal matrix. In particular, $B(\mathfrak{g})$ is equivalent to a diagonal matrix. Now, since the rows of $B(\mathfrak{g})$ are labeled by elements of $D(\mathfrak{g})\subset \mathfrak{h},$ the entries in the column labeled by $x\in T(\mathfrak{g})$ are all multiples of $x$. Thus, there exists a basis $\{d_1,e_1,\hdots,d_n,e_n\}$  for $\mathfrak{g}$ where $\{e_1,\hdots,e_n\}=T(\mathfrak{g})$, the $d_i$ are linear combinations of elements in $D(\mathfrak{g})$, and $[d_i,e_i]=e_i$, for $i=1,\hdots,n$; we denote this Lie algebra by $\Phi_n$. Therefore, every Frobenius, two-step solvable Lie poset algebra of dimension $2n$ is isomorphic to $\Phi_n$. The result follows.
\end{proof}

\begin{remark}
In \textup{\textbf{\cite{pA}}}, it is shown that height-one posets whose Hasse diagrams are connected, acyclic, bipartite graphs generate Frobenius, type-A Lie poset algebras. Since any such connected, acyclic, bipartite graph can be paired with a tree in such a way that non-isomorphic trees get paired with non-isomorphic posets, we find that there are at least $\lfloor\frac{n^{n-2}}{n!}\rfloor$ Frobenius, two-step solvable, $n$-dimensional type-A Lie poset algebras corresponding to non-isomorphic posets. Thus, Theorem~\ref{thm:iso} is not trivial.
\end{remark}

%%%%%%%%%%%%%%%%%%%%%%%%%%%%%%%%%%%%%%%%%
\section{Rigidity}
%%%%%%%%%%%%%%%%%%%%%%%%%%%%%%%%%%%%%%%%%
In this section, we show $\Phi_n$ is absolutely rigid; that is, has no infinitesimal deformations.  By Theorem~\ref{thm:iso}, it will follow that all Frobenius, two-step solvable Lie poset algebras cannot be deformed.

To set the context for our rigidity result, we recall some basic facts from the (infinitesimal) deformation theory of Lie algebras.
Recall the standard Chevalley-Eilenberg cochain complex $(C^{\bullet}(\mathfrak{g},\mathfrak{g}),\mathfrak{g})$ of $\mathfrak{g}$ with coefficients in the $\mathfrak{g}$-module $\mathfrak{g}$; that is, $C^n(\mathfrak{g},\mathfrak{g})$ consists of forms $F^n:\bigwedge_{i=1}^n\mathfrak{g}\to \mathfrak{g}$ satisfying $\delta^2 F^n=0$, where 
\begin{align}
\delta F^n(g_1,\hdots,g_{n+1})=\sum_{i=1}^{n+1}(-1)^{i+1}[&g_i, F^n(g_1,\hdots,\hat{g}_i,\hdots,g_{n+1})]\nonumber \\ 
&+\sum_{1\le i<j\le n+1}(-1)^{i+j}F^n([x_i,x_j],x_1,\hdots,\hat{x}_i,\hdots,\hat{x}_j,\hdots,x_{n+1}).\nonumber
\end{align}
In this setting, $Z^n(\mathfrak{g},\mathfrak{g})=\ker(\delta)\cap C^n(\mathfrak{g},\mathfrak{g})$, $B^n(\mathfrak{g},\mathfrak{g})=Im(\delta)\cap C^n(\mathfrak{g},\mathfrak{g})$, and $H^n(\mathfrak{g},\mathfrak{g})=Z^n(\mathfrak{g},\mathfrak{g})/ B^n(\mathfrak{g},\mathfrak{g})$.  These comprise, respectively, the \textit{n-cocycles}, \textit{n-coboundaries}, and $n^{th}$ \textit{cohomology group} of $\mathfrak{g}$ with coefficients in $\mathfrak{g}$. Up to equivalence, the infinitesimal deformations of $\mathfrak{g}$ may be regarded as elements of 
$H^2 ( \mathfrak{g}, \mathfrak{g})$ with the obstructions to their propagation to higher-order deformations lying in $H^3 ( \mathfrak{g}, \mathfrak{g})$ (see \textbf{\cite{F}} and \textbf{\cite{Fox}}).  If each element of $H^2 ( \mathfrak{g}, \mathfrak{g})$ is obstructed, then $\mathfrak{g}$
is called \textit{rigid}, and if $H^2 ( \mathfrak{g}, \mathfrak{g})=0$, then 
$\mathfrak{g}$ is said to be  \textit{absolutely rigid}.

\begin{theorem}\label{thm:h2}
$H^2(\Phi_n,\Phi_n)=0$.
\end{theorem}

\begin{proof}
Let $F^2\in C^2(\Phi_n,\Phi_n)$ be defined by $F^2(x,y)=\sum_{k=1}^n(f^{x,y}_{d_k}d_k+f^{x,y}_{e_k}e_k)$, for $x,y\in\{d_1,e_1,\hdots,d_n,e_n\}$ such that $x\neq y$. Solving $\delta F^2=0$, we find that if $F^2\in Z^2(\Phi_n,\Phi_n)$ and $i\neq j,$ then
\begin{itemize}
    \item $F^2(e_i,e_j)=f_{e_i}^{e_i,e_j}e_i+f_{e_j}^{e_i,e_j}e_j$;
    \item $F^2(d_i,d_j)=f_{e_i}^{d_i,d_j}e_i+f_{e_j}^{d_i,d_j}e_j$;
    \item $F^2(d_i,e_j)=f_{e_i}^{d_i,e_j}e_i+f_{e_j}^{d_i,e_j}e_j$; and
    \item $F^2(d_i,e_i)=f_{e_i}^{d_i,e_i}e_i+f_{d_i}^{d_i,e_i}d_i-\sum_{k\neq i}(f_{e_k}^{d_k,e_i}e_k+f_{e_k}^{e_i,e_k}d_k).$
\end{itemize}
Let $F^1\in C^1(\Phi_n,\Phi_n)$ be defined by $F^1(x)=\sum_{k=1}^n(f^{x}_{d_k}d_k+f^{x}_{e_k}e_k)$, for $x\in\{d_1,e_1,\hdots,d_n,e_n\}$. Setting
\begin{itemize}
    \item $f_{d_j}^{e_i}=f_{e_j}^{e_i,e_j}$,
    \item $f_{e_i}^{d_j}=f_{e_i}^{d_i,d_j}$,
    \item $f_{e_i}^{e_j}=f_{e_i}^{d_i,e_j}$,
    \item $f_{d_i}^{e_i}=-f_{d_i}^{d_i,e_i}$,
    \item $f_{d_j}^{d_i}=f^{d_i,e_j}_{e_j}$, for all $i,j$,
    \item and all remaining terms equal to 0,
\end{itemize}
we get $\delta F^1=F^2$. The result follows.  
\end{proof}

%%%%%%%%%%%%%%%%%%%%%%%%%%
\section{Epilogue}
%%%%%%%%%%%%%%%%%%%%%%%%%%
The original motivation for this article was to understand if the type-A cohomological result of equation (\ref{thm:CoG}) could be succinctly extended to the other classical types -- this owing to the recent introduction of the definitions of Lie poset algebras in types B, C, and D (see \textbf{\cite{pBCD}} and Definition \ref{def:bcd} below).

It is evident from (\ref{thm:CoG}) that type-A Lie poset algebras are simplicial in nature -- the chains of the poset which generates a type-A Lie poset algebra conveniently furnishing the simplicial object $\Sigma$. In the other classical cases, the analogously defined $\Sigma$ fails to satisfy (\ref{thm:CoG}). To see this, we introduce the definitions of the posets of type C, D, and B. These define generating posets for Lie posets of type C, D, and B, respectively.

\begin{definition}\label{def:bcd}
A type-C poset is a poset $\mathcal{P}=\{-n,\hdots,-1,1,\hdots, n\}$ such that
\begin{enumerate}
	\item if $i\preceq_{\mathcal{P}}j$, then $i\le j$; and
	\item if $i\neq -j$, then $i\preceq_{\mathcal{P}}j$ if and only if $-j\preceq_{\mathcal{P}}-i$.
\end{enumerate}
A type-D poset is a poset $\mathcal{P}=\{-n,\hdots,-1,1,\hdots, n\}$ satisfying 1 and 2 above as well as 
\begin{enumerate}
    \setcounter{enumi}{2}
    \item $-i\npreceq_{\mathcal{P}}i$.
\end{enumerate}
A type-B poset is a poset $\mathcal{P}=\{-n,\hdots,-1,0,1,\hdots, n\}$ satisfying 1 through 3 above. 
\end{definition}

\begin{Ex}\label{ex:typeBCD}
The poset $\mathcal{P}$ on $\{-3, -2, -1, 1, 2, 3\}$ defined by $-1\preceq 2,3$; $-2\preceq 1,3$; and $-3\preceq 1,2$ forms a type-C poset. The Hasse diagram of $\mathcal{P}$ is illustrated in Figure~\ref{fig:tBCD} \textup(left\textup).  The matrix form defining the corresponding type-C Lie poset algebra, denoted $\mathfrak{g}_C(\mathcal{P})$, is illustrated in Figure~\ref{fig:tBCD} \textup(right\textup). 
Clearly, $\mathfrak{g}_C(\mathcal{P})$ is two-step solvable. Furthermore, using Theorem~\ref{thm:commat}, it is straightforward to show that $\mathfrak{g}_C(\mathcal{P})$ is Frobenius. Thus, by Theorem~\ref{thm:h2}, $H^2(\mathfrak{g}_C(\mathcal{P}),\mathfrak{g}_C(\mathcal{P}))=0$. If equation \textup{(\ref{thm:CoG})} applied to $\mathfrak{g}_C(\mathcal{P})$, then $H^2(\mathfrak{g}_C(\mathcal{P}),\mathfrak{g}_C(\mathcal{P}))\neq 0$; this follows as the corresponding simplicial object would be the Hasse diagram of $\mathcal{P}$ which is homotopy equivalent to a circle.
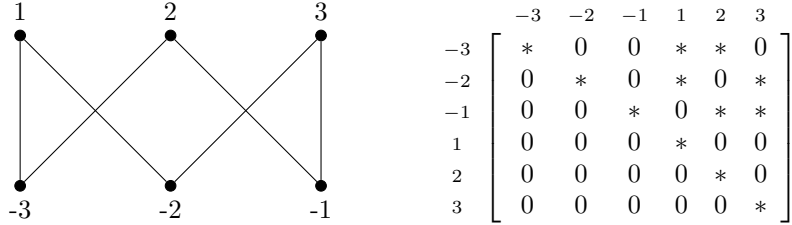
\begin{figure}[H]
$$\begin{tikzpicture}
	\node (-3) at (0, 0) [circle, draw = black, fill = black, inner sep = 0.5mm, label=below:{-3}]{};
	\node (-2) at (2, 0)[circle, draw = black, fill = black, inner sep = 0.5mm, label=below:{-2}] {};
	\node (-1) at (4, 0) [circle, draw = black, fill = black, inner sep = 0.5mm, label=below:{-1}] {};
	\node (1) at (0, 2) [circle, draw = black, fill = black, inner sep = 0.5mm, label=above:{1}]{};
	\node (2) at (2, 2)[circle, draw = black, fill = black, inner sep = 0.5mm, label=above:{2}] {};
	\node (3) at (4, 2) [circle, draw = black, fill = black, inner sep = 0.5mm, label=above:{3}] {};
	\node (5) at (7,1) {
  \kbordermatrix{
    & -3 & -2 & -1 & 1 & 2 & 3 \\
   -3 & * & 0 & 0 & * & * & 0  \\
   -2 & 0 & * & 0 & * & 0 & * \\
   -1 & 0 & 0 & * & 0 & * & * \\
   1 & 0 & 0 & 0 & * & 0 & 0 \\
   2 & 0 & 0 & 0 & 0 & * & 0 \\
   3 & 0 & 0 & 0 & 0 & 0 & * \\
  }
};
    \draw (1)--(-3)--(2);
    \draw (1)--(-2)--(3);
    \draw (2)--(-1)--(3);
    \addvmargin{1mm}
\end{tikzpicture}$$
\caption{Hasse diagram of $\mathcal{P}$ (left) and $\mathfrak{g}_C(\mathcal{P})$ (right)}\label{fig:tBCD}
\end{figure}
\end{Ex}

\noindent
Example~\ref{ex:typeBCD} suggest two questions:

\begin{itemize}
    \item Is there a more general isomorphism result (per the Classification theorem) relating all Lie poset algebras allowing for the use of the type-A result in (\ref{thm:CoG}).
    \item At the cohomological level, are all Lie poset algebras essentially simplicial or is this unique to the type-A setting?  If there is an ``appropriate'' simplicial object, what is it, and is it related to the nerve of a generating poset?
\end{itemize}

\section{Appendix-Spectrum}

\begin{theorem}
$\Phi_n$ has a binary spectrum.
\end{theorem}
\begin{proof}
Let $e_i^*$, for $i=1,\hdots,n$, denote the functional on $\Phi_n$ which returns the coefficient of $e_i$. A straightforward exercise in linear algebra shows that $f=\sum_{i=1}^ne^*_i$ is a Frobenius functional on $\Phi_n$ with principal element $\hat{f}=\sum_{i=1}^nd_i$. Calculating $[\hat{f},d_i]$ and $[\hat{f},e_i]$, for $i=1,\hdots,n$, establishes the result.
\end{proof}

\end{document}